\documentclass[11pt]{article}
\usepackage{amsmath,amssymb,amsthm,xcolor,graphicx}
\usepackage[pagewise]{lineno}
\usepackage{accents}
\usepackage[unicode,breaklinks=true,colorlinks=true,citecolor = {magenta}]{hyperref}
\usepackage{esint}  
\usepackage{listings}  

\DeclareSymbolFont{yhlargesymbols}{OMX}{yhex}{m}{n}
\DeclareMathAccent{\wideparen}{\mathord}{yhlargesymbols}{"F3}

\numberwithin{equation}{section}
\newtheorem{theorem}{Theorem}[section]

\newtheorem{lemma}[theorem]{Lemma}

\newtheorem{definition}{Definition}[section]

\theoremstyle{remark}
\newtheorem{remark}[theorem]{Remark}
\renewcommand{\div}{\mathop{\rm div}\nolimits}

\newcommand{\supp} {\mathop{\mathrm{supp}}\nolimits}

\newcommand{\EQ}[1]{\begin{equation} #1 \end{equation}}

\usepackage{mathtools} 

\begin{document}
\title{Existence of global weak solutions to the Navier–Stokes equations in weighted spaces}
\author{Misha Chernobai\footnote{Department of Mathematics, UBC, Canada, mchernobay@gmail.com}}
\date{\today}
\maketitle

 \abstract{We obtain a global existence result for the three-dimensional Navier-Stokes equations
with a large class of initial data allowing growth at spatial infinity. Our work is a continuation of the results \cite{BKT}, \cite{BK} and proves global existence of suitable weak solutions with initial data in different weighted spaces as well as eventual regularity.}

\section{Introduction}
Let $(u,p)$ be a solution the three dimensional Navier-Stokes equations in the sense of distributions:
\EQ{
\left\{\gathered
\partial_t u-\Delta u+(u\cdot\nabla)u+\nabla p=0,
\\
\div u=0
\endgathered\right. ~~\text{in}~\Bbb R^3\times (0,T)
\label{5NS}
}
For $T>0$ and initial data $u(x,0)=u_0(x)$. One of the first results about global in time existence of weak solutions was done by J.Leray in \cite{Leray} with divergence free initial data $u_0\in L^2(\Bbb R^3)$. Later it was improved by Lemari\'{e}-Rieusset in \cite{LR} for local initial data $u_0\in L^2_{uloc}$, the local $L^2_{uloc}$ space is defined as following:
\EQ{
\|u_0\|_{L^2_{uloc}}=\sup_{x\in \Bbb R^3}\|u_0\|_{L^2(B(x,1))}<\infty .
}
The focus of this work is on finding a different class for initial data, for which we can prove global in time existence, for that we need to define local-energy solutions. We will use the following notations, for any cube $Q\in \Bbb R^3$ with radius $r$ we will denote $Q^*$ a slightly bigger cube with the same center and radius $5r/4$ and $Q^{**}$ is similarly a small extension of $Q^*$. Next we introduce the definition of local-energy solutions:
\begin{definition}
  A vector field $u\in L^2_{loc}(\Bbb R^3\times (0,T))$, where $T>0$ is a local energy solution to \eqref{5NS} with initial data $u_0\in L^2_{loc}(\Bbb R^3)$ if the following holds:
  \begin{itemize}
    \item $u\in\underset{R>0}{\cap}L^{\infty}(0,T;L^2(B_R(0))),~\nabla u\in L^2_{loc}(\Bbb R^3\times[0,T))$,
    \item there is $p\in L^{\frac32}_{loc}(\Bbb R^3\times[0,T))$ such that ${u,p}$ is a solution to NSE \eqref{5NS} in a sense of distributions,
    \item For all compacts $\Omega\subset\Bbb R^3$ we have $u(t)\underset{t\rightarrow0+}{\rightarrow} u_0$ in $L^2(\Omega)$,
    \item $u$ is a Caffarelli-Kohn-Nirenberg solution, for all $\xi \in C^{\infty}_0(\Bbb R^3\times (0,T))$, $\xi\ge0,$
    \EQ{\gathered
    2\int\int|\nabla u|^2\xi ~dx~dt\le\int\int |u|^2(\partial_t\xi+\Delta \xi)~dx~dt+
    \\
    \int\int(|u|^2+2p)(u\cdot \nabla\xi)~dx~dt,
    \endgathered}
    \item the function $t\mapsto\int u(x,t)w(x)~dx$ is continuous on $[0,T)$ for any $w\in L^2(\Bbb R^3)$ with a compact support,
    \item for every cube $Q\subset \Bbb R^3$, there exists $p_Q(t)\in L^{\frac 32}(0,T)$ such that  for $x\in Q^*$ and $0<t,T$,
    \EQ{\gathered
    p(x,t)-p_Q(t)=-\frac13|u(x)|^2+p.v.\int_{y\in Q^{**}}K_{ij}(x-y)(u_i(y,s)u_j(y,s))~dy
    \\
    +\int_{y\notin Q^{**}}(K_{ij}(x-y)-K_{ij}(x_Q-y))(u_i(y,s)u_j(y,s))~dy,
    \endgathered}
    where $x_Q$ is the center of $Q$ and $K_{ij}=\partial_i\partial_j(4\pi|y|)^{-1}$.
  \end{itemize}
\end{definition}
This definition was used in \cite{BKT} and differs from previous definitions in \cite{LR, KiSe, JiaSverak-minimal, BT8, RS} since we only require initial data in $L^2_{loc}$ and not $u_0\in L^2_{uloc}$. Nevertheless, local energy solutions first were proven to exist locally in time for initial data in $L^2_{uloc}(\Omega)$ for $\Omega=\Bbb R^3$ or $\Bbb R^3_+$. Most of the results for global existence included a decay assumption on the data, see \cite{LR, KiSe} for $v_0\in E^2$. One of the examples of such conditions is:
\EQ{
\lim_{R\rightarrow\infty}\sup_{x_0\in\Bbb R^3}\frac{1}{R^2}\int\limits_{B_R(x_0)}|u_0|^2~dx=0.
}
Later global existence for non-decaying data in uniform space $L^2_{uloc}$ was done in two dimensional case by Bason \cite{Basson}. For initial data in $L^2_{loc}\setminus L^2_{uloc}$ Fernandez-Dalgo and Lemari\'{e}-Rieusset \cite{FDLR} constructed global solutions with the following condition:
\EQ{
\int_{\Bbb R^3}\frac{|u_0(x)|^2}{(1+|x|)^2}~dx<\infty.
}
Bradshaw and Kukavica \cite{BK} constructed local in time solutions for initial data in the space ${\mathring{M}}^{p,q}_{\mathcal{C}}$ which is related to our project and will be defined later. In \cite{BKT} with Tai-Peng Tsai they proved global existence in class ${\mathring{M}}^{p,q}_{\mathcal{C}}$ , this class strictly includes the initial data from \cite{BT8, FDLR}. In our paper we want to extend this result to different weighted spaces of initial data $F^q_r$.

To define spaces ${\mathring{M}}^{p,q}_{\mathcal{C}},F^q_r$ we need to construct a certain family of cubes in $\Bbb R^3$ that was defined in \cite{BKT}. For $n\in \Bbb N$, we denote $S_0=\{x:~|x_i|\le 2;~i=1,2,3\}$ and $R_n=\{x:~|x_i|\le 2^n; i=1,2,3\}$. Let $S_n=\bar{ R_{n+1}\setminus R_n}$ for $n\in \Bbb N\setminus\{1\}$ and $S_1=\bar{R_2\setminus S_0}$. Then $|S_n|=56\cdot 2^{3n}$. We separate $S_0=\cup_1^{64} Q_0^{k_0}$ into 64 cubes with side-length $1$ and $S_n=\cup_{k_n=1}^{56}Q_n^{k_n}$ into 56 cubes with side-length $2^n$, the set of all these cubes we denote as $\mathcal{C}=\cup_{n=0}^{\infty}\cup_{k_n}Q_n^{k_n}$. This set will satisfy the following properties
\begin{itemize}
  \item Side-length of each cube is proportional to the distance to the origin.
  \item Adjacent cubes have proportional volume.
  \item The distance between cubes $Q,Q'$ is proportional to $\min\{|Q|^{\frac13}, |Q'|^{\frac13}\}$.
  \item The amount of cubes $Q'$, such that $|Q'|<|Q|$ is bounded above by $|Q|^{\frac13}$
\end{itemize}

Using set $\mathcal{C}$ we reintroduce the definition of $M^{p,q}_{\mathcal{C}}$ and $\mathring{M}^{p,q}_{\mathcal{C}}$ from \cite{BKT}:
\begin{definition}
  Let $p\in [1,\infty), q\ge 0$. We have $f\in M^{p,q}_{\mathcal{C}}$ if
  \EQ{
  \|f\|^p_{M^{p,q}_{\mathcal{C}}}:=\sup_{Q\in\mathcal{C}}\frac{1}{|Q|^{\frac{q}{3}}}\int_Q|f(x)|^p~dx<\infty.
  }
  If $f\in M^{p,q}_{\mathcal{C}}$ and satisfies the following condition then $f\in\mathring{M}^{p,q}_{\mathcal{C}}$:
  \EQ{
  \frac{1}{|Q|^{\frac{q}{3}}}\int_{Q}|f|^p~dx\rightarrow 0,~\text{as}~|Q|\rightarrow\infty,~Q\in \mathcal{C}
  }
  \end{definition}
Note that global existence and eventual regularity result  \cite{BKT} considers initial data from $\mathring{M}^{2,2}_{\mathcal{C}}$ and improves the constructions in \cite{BT8, BK, FDLR}. The goal of our project is to further improve the result of Bradshaw, Kukavica and Tsai \cite{BKT} for another set of initial data, using the family of cubes $\mathcal{C}$ let us introduce the following spaces:
\begin{definition}
Let $q\le 2$ and $r>2$, we define spaces $F^q_r$:
\EQ{
F^q_r=\Big\{f\in L^2_{loc}(\Bbb R^3)~\Big|~\sum_{n,k_n}\Big(\frac{1}{|Q_n^{k_n}|^{\frac{q}{3}}}\int_{Q_n^{k_n}}|f|^2~dx\Big)^{\frac r2}<\infty\Big\}
}
With the norm $\|f\|^r_{F^q_r}=\sum_{n,k_n}^{\infty}\Big(\frac{1}{|Q_n^{k_n}|^{\frac{q}{3}}}\int_{Q_n^{k_n}}|f|^2~dx\Big)^{\frac{r}2}$
\end{definition}
Both spaces $M^{p,q}_{\mathcal{C}}$ and $F^q_r$ are equivalent to certain Herz spaces. Let $A_k=B_{2^{k+1}}\setminus B_{2^k}$ where $B_{2^k}$ is a ball in $\Bbb R^3$ with radius $2^k$. Let $n\in\Bbb N,~s\in\Bbb R$ and $p,q\in(0,\infty]$, we define the non-homogeneous Herz space $K_{p,q}^s(\Bbb R^n)$ as a space of functions $f\in L^p_{loc}(\Bbb R^n\setminus\{0\})$ with the following finite norm:
\EQ{
\|f\|_{K^s_{p,q}}=\left\{\gathered
\Big(\sum_{k\in\Bbb N_0}2^{ksq}\|f\|^q_{L^p(A_k)}\Big)^{\frac{1}{q}} ~\text{if}~q<\infty,
\\
\sup_{k\in\Bbb N_0}2^{ks}\|f\|_{L^p(A_k)} ~\text{if}~q=\infty.
\endgathered\right.
}
Then the space $M^{p,q}_{\mathcal{C}}$ is equivalent to Herz space $K^{-q}_{p,\infty}$ and $F^q_r$ is equivalent to $K^{-q}_{2,r}, ~r<+\infty$. Local in time existence of mild solutions with initial data from some Herz spaces was proven by Tsutsui \cite{Tsutsui} (in case $p>3$). Global existence of local energy solutions for the case $K^{-2}_{2,\infty}$ follows from \cite{BKT}. Our goal is to extend the construction of global in time local energy solutions for initial data in Herz spaces $K^{-q}_{2,r}$.

The following is our main theorem on global existence of the solutions with initial data from $F^q_r$:
\begin{theorem}
   Let $q\le 2,~ r>2$ and assume $u_0\in F^q_r$ is divergence free. Then there exists $u:\Bbb R^3\times(0,\infty)\rightarrow\Bbb R^3$ and $p:\Bbb R^3\times(0,\infty)\rightarrow\Bbb R$ so that $(u,p)$ is a local energy solution to the Navier-Stokes equations on $\Bbb R^3\times (0,\infty)$ corresponding to initial data $u_0$.
\end{theorem}

Additionally we will use the following sequence of norms $F^{q,m}_r$ using set of cubes $\mathcal{C}$.

\begin{definition}
For $m\in\Bbb N$ we denote $Q_{0,m}=R_{m}$ and for any $n\ge 1,k_n=1...56,~Q_{n,m}^{k_n}=Q_{m+n}^{k_{n}}$, where $Q_n^{k_n}$ are cubes of family $\mathcal{C}$. With respect to these sets we can denote norms
$$
\|f\|^r_{F^{q,m}_r}=\sum_{n,k_n}\Big(\frac{1}{|Q_{n,m}^{k_n}|^{\frac{q}{3}}}\int_{Q_{n,m}^{k_n}}|f|^2~dx\Big)^{\frac{r}2}.
$$
\end{definition}

We will denote $\mathcal{C}_m$ family of cubes in this definition that corresponds to space $F^{q,m}_r$. The following lemma shows the relation between spaces $F^q_r$ and $F^{q,m}_r$:
\begin{lemma}\label{5intialdata}
  Let $u\in F^q_r$, then we have the following two properties
  \begin{itemize}
  \item $\|u\|_{F^{q,n}_{r}}\underset{n\rightarrow\infty}\rightarrow 0$.
    \item For any $n$ $u\in F^{q,n}_{r}$ and there exists constant $C$ independent on $n$ such that $$\|u\|_{F^{q,n}_{r}}\le C\|u\|_{F^{q}_{r}}$$.
  \end{itemize}
\end{lemma}
\begin{proof}
Let us start with the first property, since the second one will follow from the proof. For the first property we only have to prove the convergence of the first term, corresponding to $Q_{0,n}$, since we know that $u\in F^q_r$ and the rest of the sum converges to $0$:
  $$
  \frac{1}{2^{nq}}\int_{B_{2^n}}|u|^2~dx\underset{n\rightarrow\infty}\rightarrow0.
  $$
  We denote $a_{m,k_m}=\Big(\frac{1}{2^{mq}}\int_{Q_m^{k_m}}|u|^2~dx\Big)^{\frac{r}2}$, for $m\le n$, $k_m\le 64$, here $Q_m^{k_m}$ are cubes in the family $\mathcal{C}$ with radius $2^m$, the following estimate holds
  \EQ{\label{51.12}\gathered
  \frac{1}{2^{nq}}\int_{B_{2^n}}|u|^2~dx=\frac{1}{2^{nq}}\sum_{m=0}^{n-1}\sum_{k_m}\int_{Q_{m}^{k_m}}|u|^2~dx=\frac{1}{2^{nq}}
  \sum_{m=0}^{n-1}\sum_{k_m}2^{mq}{a_{m,k_m}}^{\frac{2}{r}}=
  \\
  =\sum_{m=0}^{n-1}\sum_{k_m}\frac{1}{2^{(n-m)q}}{a_{m,k_m}}^{\frac{2}{r}}\le \sum_{m=0}^{l-1}\sum_{k_m}\frac{1}{2^{(n-m)q}}{a_{m,k_m}}^{\frac{2}{r}}+\sum_{m=l-1}^{n-1}\sum_{k_m}\frac{1}{2^{(n-m)q}}{a_{m,k_m}}^{\frac{2}{r}}.
  \endgathered}
  Where $l\le n-1$ is an integer that we will choose later. Next we will estimates last two terms of \eqref{51.12} separately
  \EQ{
  \sum_{m=0}^{l-1}\sum_{k_m}\frac{1}{2^{(n-m)q}}{a_{m,k_m}}^{\frac{2}{r}}\le \frac{C}{2^{(n-l)q}}\sup_{n,k_n}|a_{n,k_n}|^{\frac{2}{r}}
  }
  For the last term we will use Holder inequality:
  \EQ{\label{5113}\gathered
  \sum_{m=l-1}^{n-1}\sum_{k_m}\frac{1}{2^{(n-m)q}}{a_{m,k_m}}^{\frac{2}{r}}\le \Big(\sum_{m=l-1}^{n-1}\sum_{k_m}\frac{1}{2^{\frac{r(n-k)q}{r-2}}}\Big)^{\frac{r-2}{r}}\Big(\sum_{m=l-1}^{\infty}\sum_{k_m}a_{m,k_m}\Big)^{\frac{2}{r}}
  \le
  \\
  \le C(q,r)\Big(\sum_{m=l-1}^{\infty}\sum_{k_m}a_{m,k_m}\Big)^{\frac{2}{r}}.
  \endgathered}
  Next we choose $n$ sufficiently large and $l\approx\frac{n}{2}$, with this both terms will be sufficiently small, which implies the convergence:
   $$
  \frac{1}{2^{nq}}\int_{B_{2^n}}|u|^2~dx\underset{n\rightarrow\infty}\rightarrow 0
  $$
  To check the second statement we again only need to estimate the first term for which we will use \eqref{51.12} and apply Holder inequality similarly to \eqref{5113}
  \EQ{\gathered
  \frac{1}{2^{nq}}\int_{B_{2^n}}|u|^2~dx=\sum_{m=0}^{n-1}\sum_{k_m}\frac{1}{2^{(n-m)q}}{a_{m,k_m}}^{\frac{2}{r}}\le
  \\
  \Big(\sum_{m=1}^{n-1}\sum_{k_m}\frac{1}{2^{\frac{r(n-k)q}{r-2}}}\Big)^{\frac{r-2}{r}}\Big(\sum_{m=1}^{\infty}\sum_{k_m}a_{m,k_m}\Big)^{\frac{2}{r}}\le
  C(q,r)\Big(\sum_{m=1}^{\infty}\sum_{k_m}a_{m,k_m}\Big)^{\frac{2}{r}}=C\|u\|_{F^q_r}^2.
  \endgathered}
  Here $C=C(q,r)$ is independent of $n$ which finishes the proof.
\end{proof}
Using Lemma \ref{5intialdata} and a priori bounds we can establish the following regularity result:
\begin{theorem}(Eventual Regularity)\label{5Regularity}
    Take $q\le 1$ and $u_0\in F^q_r$ as divergence-free initial data, and let $(u,p)$ be a local energy solution on $\Bbb R^3\times (0,\infty)$ corresponding to $u_0$. Assume additionally that:
  \EQ{
  \|u(\cdot,t)\|_{F_{r}^{q}}^r+\sum_{m=1}^{\infty}\Big(\frac{1}{|Q_m|^{\frac{q}{3}}}\int_0^t\int_{Q_m}|\nabla u|^2~dx~dt\Big)^{\frac{r}2}< \infty,
  }
  for all $t<\infty$. Then for any $\delta>0$, there exists a time $\tau$ depending on $u_0,\delta$ so that $u$ is smooth on
  \EQ{
  \{(x,t),t>\max\{\delta|x|^2,\tau\}\}
  }
\end{theorem}

The paper organized as follows. In Section \ref{5pressuresection} we prove a bound for pressure term, using this bound in Section \ref{5leisection} we establish a priori estimate for the solution $u$ in spaces $F^{q,m}_r$. In Section \ref{5GEsection} we prove global existence of local energy solutions and in Section \ref{5regularsection} we prove Theorem \ref{5Regularity}.

\medskip\noindent
{\bf Acknowledgement.} The research of Misha Chernobai was partially supported by Natural Sciences and Engineering Research Council of Canada (NSERC) grants RGPIN-2018-04137 and RGPIN-2023-04534. Author expresses gratitude to Professor Tai-Peng Tsai for suggesting the statement of the problem.

\section{Pressure estimate}\label{5pressuresection}
We use a similar representation of pressure as in \cite{BKT}, in this Section we fix some integer $N$ and family of cubes corresponding to the space $F^{q,N}_r$, take $Q$ as one of the cubes in this family:
$$
p(x,t)-p_Q(t)=(G_{ij}^Qu_iu_j)(x,t)~,\forall x\in Q^*,
$$
Where
\EQ{\label{Gij}\gathered
G_{ij}f(x)=-\frac13\delta_{ij}f(x)+\textrm{p.v.}\int_{y\in Q^{**}}K_{ij}(x-y)f(y)~dy+
\\
+\int_{y\notin Q^{**}}(K_{ij}(x-y)-K_{ij}(x_Q-y))f(y)~dy
\endgathered}
The following lemma will be the estimate of pressure in spaces $F^{q,N}_r$

\begin{lemma}\label{5pressure}
Let $q\le 2$ and $(u,p)$ be a local energy solution and associated pressure on $\Bbb R^3\times[0,T]$. Then for any integer $N$ and cube $Q$ in a family of cubes corresponding to the space $F^{q,N}_r$ we have the following estimate:
\EQ{\gathered
\frac{1}{|Q|^{\frac{q+1}{3}}}\int_0^t\int_Q |p(x,t)-(p)_{Q(t)}|^{\frac32}~dx~ds\le
\\
\frac{C}{|Q|^{\frac{5-q}{6}}}\int_0^t\Big(\sum_{Q'\subset (Q^{**})^{c}}\Big(\frac{1}{|Q'|^{\frac{q}{3}}}\int_{Q'}|u|^2~dy\Big)^{\frac{r}2}\Big)^{\frac{3}{r}}~ds+\frac{C}{|Q|^{\frac{q+1}3}}\int_0^t\int_{Q^{**}}|u|^3~dx~ds,
\endgathered}
where $Q^{**}$ is a cube of radius $\frac{5}{4}|Q|^{\frac{1}{3}}$ and the same center as $Q$.
\end{lemma}
\begin{proof}

We separate $G_{ij}$ from \eqref{Gij} into two terms:
 $$
 I_{near}=-\frac13\delta_{ij}f(x)+\textrm{p.v.}\int_{y\in Q^{**}}K_{ij}(x-y)f(y)~dy,
  $$
  $$
  I_{far}=\int_{y\notin Q^{**}}(K_{ij}(x-y)-K_{ij}(x_Q-y))f(y)~dy.
 $$
Next we apply Calderon-Zigmund estimates for the term $I_{near}$:
$$
\frac{1}{|Q|^{\frac13}}\int_0^t\int_{Q^*}|I_{\textrm{near}}|^{\frac32}~dx~ds\le \frac{C}{|Q|^{\frac13}}\int_0^t\int_{Q^{**}}|u|^3~dx~ds.
$$
This term is similar to the one we will have in local energy inequality and will be estimated in the next Section. Now we will estimate the other term  $I_{\textrm{far}}$:
\EQ{\gathered
I_{\textrm{far}}\le C |Q|^{\frac{1}{3}}\int_{\Bbb R^3\setminus Q^{**}}\frac{1}{|x-y|^4}|u|^2~dy\le
\\
\le C\sum_{Q'\in S_1}|Q|^{\frac{1}{3}}\int_{Q'\cap (Q^{**})^{c}}\frac{1}{|x-y|^4}|u|^2~dy+ C\sum_{Q'\in S_2}|Q|^{\frac{1}{3}}\int_{Q'\cap (Q^{**})^{c}}\frac{1}{|x-y|^4}|u|^2~dy
\endgathered}
Where $S_1$ is the subset of cubes with $|Q'|\le 8|Q|$ and $S_2$ are cubes with $|Q'|>8|Q|$. There is only finite amount of cubes in $S_1$ and we can also assume for them that $|Q|^{\frac{1}{3}}\approx |x_Q-x_{Q'}|$ due to construction of our family of cubes $\mathcal{C}$. Moreover, if $x\in Q$ and $y\in Q',~Q'\in S_1$ then $|x-y|\approx |x_Q-x_{Q'}|\approx |Q|^{\frac{1}{3}}$. Hence,
\EQ{\gathered
\sum_{Q'\in S_1}|Q|^{\frac{1}{3}}\int_{Q'\cap (Q^{**})^{c}}\frac{1}{|x-y|^4}|u|^2~dy\le \sum_{Q'\in S_1}\frac{1}{|Q|}\int_{Q'\cap (Q^{**})^{c}}|u|^2~dy\le
\\
\le \frac{1}{|Q|}\Big(\sum_{Q'\in S_1}\Big(\frac{1}{|Q'|^{\frac{q}{3}}}\int_{Q'}|u|^2~dy\Big)^{\frac{r}{2}}\Big)^{\frac{2}{r}}\Big(\sum_{Q'\in S_1}|Q'|^{\frac{qr}{3(r-2)}}\Big)^{\frac{r-2}{r}}\le
\\
\le \frac{C}{|Q|^{1-\frac{q}{3}}}\Big(\sum_{Q'\in S_1}\Big(\frac{1}{|Q'|^{\frac{q}{3}}}\int_{Q'}|u|^2~dy\Big)^{\frac{r}{2}}\Big)^{\frac{2}{r}}
\endgathered}
On the other hand for $Q'\in S_2$ we have a different relation for $x\in Q,y\in Q':~|x-y|\approx |Q'|^{\frac{1}{3}}$
\EQ{\gathered
 \sum_{Q'\in S_2} |Q|^{\frac{1}{3}}\int_{Q'}\frac{1}{|x-y|^4}|u|^2~dy
\le C |Q|^{\frac{1}{3}}\sum_{Q'\in S_2}\frac{1}{|Q'|^{\frac{4}{3}}}\int_{Q'}|u|^2~dy\le
\\ \le C |Q|^{\frac{1}{3}}\Big(\sum_{Q'\in S_2}\Big(\frac{1}{|Q'|^{\frac{q}{3}}}\int_{Q'}|u|^2~dy\Big)^{\frac{r}2}\Big)^{\frac{2}{r}}\cdot\Big(\sum_{Q'\in S_2}\Big(\frac{1}{|Q'|^{\frac{4-q}{3}}}\Big)^{\frac{r}{r-2}}\Big)^{\frac{r-2}{r}}
\endgathered}
The last term in the product is finite only for $r>2$ and $q<4$ , since all of our cubes have radius greater than $|Q|^{\frac{1}{3}}$ we can calculate the multiplier, denote $M=\log_2(|Q|^{\frac{1}{3}})$:
\EQ{\gathered
\sum_{Q'\in S_2}\Big(\frac{1}{|Q'|^{\frac{4-q}{3}}}\Big)^{\frac{r}{r-2}}=C\sum_{n\ge M} 2^{-\frac{n(4-q)r}{3(r-2)}}=C\frac{2^{-\frac{M(4-q)r}{3(r-2)}}}{1-2^{-\frac{(4-q)r}{3(r-2)}}}=C(q,r)|Q|^{-\frac{(4-q)r}{3(r-2)}}
\endgathered}
Therefore, we get the estimate for $I_{\textrm{far}}$
\EQ{\gathered
I_{\textrm{far}}\le C |Q|^{\frac{1}{3}-\frac{(4-q)}{3}}\Big(\sum_{Q'\in S_2}\Big(\frac{1}{|Q'|^{\frac{q}{3}}}\int_{Q'}|u|^2~dy\Big)^{\frac{r}2}\Big)^{\frac{2}{r}}+
\\
\frac{C}{|Q|^{1-\frac{q}{3}}}\Big(\sum_{Q'\in S_1}\Big(\frac{1}{|Q'|^{\frac{q}{3}}}\int_{Q'}|u|^2~dy\Big)^{\frac{r}{2}}\Big)^{\frac{2}{r}}\le
\\
C |Q|^{\frac{q}{3}-1}\Big(\sum_{Q'\subset (Q^{**})^{c}}\Big(\frac{1}{|Q'|^{\frac{q}{3}}}\int_{Q'}|u|^2~dy\Big)^{\frac{r}2}\Big)^{\frac{2}{r}}.
\endgathered}
Take any cube $Q$ in family of cubes corresponding to space $F^{q,N}_r$, and we get the following estimate for $I_{\textrm{far}}$
\EQ{\gathered
\frac{1}{|Q|^{\frac{q+1}{3}}}\int_0^t\int_Q |I_{\textrm{far}}|^{\frac32}~dx~ds\le
\\
\frac{C}{|Q|^{\frac{q+1}{3}+\frac{3-q}{2}-1}}\int_0^t\Big(\sum_{Q'\subset (Q^{**})^{c}}\Big(\frac{1}{|Q'|^{\frac{q}{3}}}\int_{Q'}|u|^2~dy\Big)^{\frac{r}2}\Big)^{\frac{3}{r}}~ds=
\\
\frac{C}{|Q|^{\frac{5-q}{6}}}\int_0^t\Big(\sum_{Q'\subset (Q^{**})^{c}}\Big(\frac{1}{|Q'|^{\frac{q}{3}}}\int_{Q'}|u|^2~dy\Big)^{\frac{r}2}\Big)^{\frac{3}{r}}~ds.
\endgathered}
Combining the estimates for $I_{\textrm{far}}$ and $I_{\textrm{near}}$ we finish the proof of Lemma \eqref{5pressure}.
\end{proof}

\section{Local Energy Inequality}\label{5leisection}
In the Section we will prove a priori bound for solutions of \eqref{5NS} in spaces $F^{q,N}_r$.
\begin{lemma}\label{5aprrox}
  Assume $u_0\in F^q_r$ is divergence free and let $(u,p)$ be a local energy solution with initial data $u_0$ on $\Bbb R^3\times (0,T)$. Let us assume that for some $N$ we have the following
  \EQ{
  \|u(\cdot,t)\|_{F_{r}^{q,N}}^r+\sum_{n=1}^{\infty}\Big(\frac{1}{|Q_n|^{\frac{q}{3}}}\int_0^t\int_{Q_n}|\nabla u|^2~dx~dt\Big)^{\frac{r}2}<\infty. ~\forall t<T.
  }
  Then for all $t\in (0,T)$ we have an a priori estimate
  \EQ{
\gathered
\|u(\cdot,t)\|_{F_{r}^{q,N}}^r+\sum_{n=1}^{\infty}\Big(\frac{1}{|Q_n|^{\frac{q}{3}}}\int_0^t\int_{Q_n}|\nabla u|^2~dx~dt\Big)^{\frac{r}2}
\le  \|u(\cdot,0)\|_{F_r^{q,N}}^r +
\\
\frac{C_0t^{\frac{r-2}{2}}}{2^{rN}}\int_0^t\|u(\cdot,s)\|^r_{F_r^{q,N}}~ds+\frac{C_1t^{\frac{r-2}{2}}}{2^{rN(2-q)}}\int_0^t\|u(\cdot,s)\|^{3r}_{F_r^{q,N}}~ds
\endgathered }
Where $k_1,k_2,C_0,C_1$ are global constants that dont depend on $N$ and $Q_n$ are cubes of the family $\mathcal{C}$ corresponding to the space $F^{q,N}_r$.
\end{lemma}
\begin{proof}
Similarly to previous Section we choose arbitrary integer $N$ and denote $Q_n$ as one of the cubes in the set corresponding to the space $F^{q,N}_r$. The solution satisfies the local energy inequality for any compactly supported positive test function $\phi\le 1$ such that $\phi=1$ on $Q_n$, $\supp\{\phi\}\subset Q_n^*$ and $t\le T$:
\EQ{\label{5LEI}
\gathered
\int_{Q_n}|u|^2(x,t)~dx+\int_0^t\int_{Q_n}|\nabla u|^2~dx~dt\le \int \phi(x,0)|u(x,0)|^2~dx+
\\
\int\int_{Q_n}|u|^2\Delta\phi~dx~dt+\int\int |u|^2u\cdot\nabla\phi + 2(p-(p)_{Q_n})u\cdot\nabla \phi~dx~dt.
\endgathered }
Dividing by $|Q_n|^{\frac{q}3}$ and taking to the power $\frac{r}2$ we proceed to:
\EQ{
\gathered
\Big(\frac{1}{|Q_n|^{\frac{q}{3}}}\int_{Q_n}|u|^2(x,t)~dx\Big)^{\frac{r}2}+\Big(\frac{1}{|Q_n|^{\frac{q}{3}}}\int_0^t\int_{Q_n}|\nabla u|^2~dx~dt\Big)^{\frac{r}2}\le
\\
C\Big(\frac{1}{|Q_n|^{\frac{q}{3}}}\int \phi(x,0)|u(x,0)|^2~dx\Big)^{\frac{r}2}+C\Big(\frac{1}{|Q_n|^{\frac{q}{3}}}\int_0^t\int |u|^2\Delta\phi~dx~dt\Big)^{\frac{r}2}+
\\
C\Big(\frac{1}{|Q_n|^{\frac{q}{3}}}\int_0^t\int |u|^2u\cdot\nabla\phi + 2(p-(p))u\cdot\nabla \phi~dx~dt\Big)^{\frac{r}2}\le
\\
C\Big(\frac{1}{|Q_n|^{\frac{q}{3}}}\int_{Q_n^*}|u(x,0)|^2~dx\Big)^{\frac{r}2}+C\Big(\frac{1}{|Q_n|^{\frac{q+2}{3}}}\int_0^t\int_{Q^*_n}|u|^2~dx~dt\Big)^{\frac{r}2}+
\\
C\Big(\frac{1}{|Q_n|^{\frac{q+1}{3}}}\int_0^t\int_{Q_n^*} |u|^3~dx~dt\Big)^{\frac{r}2} +C\Big(\frac{1}{|Q_n|^{\frac{q+1}{3}}}\int_0^t\int |p-(p)_{Q_n}|^{\frac32}~dx~dt\Big)^{\frac{r}2}.
\endgathered }
Here $C=C(r)$ is a uniform constant. For the non-linear term we apply Gagliardo-Nirenberg, H\"older and Young inequalities:
\EQ{ \label{5convectiveterm}\gathered
\frac{1}{|Q_n|^{\frac13}}\int_0^t\int_{Q_n}|u|^3~dx~dt\le C(\varepsilon)|Q_n|^{q-\frac43}\int_0^t\Big(\frac{1}{|Q_n|^{\frac{q}{3}}}\int_{Q_n}|u|^2~dx\Big)^3~dt+
\\
\varepsilon\int_0^t\int_{Q_n}|\nabla u|^2~dx~ds+C|Q_n|^{\frac{q}{2}-\frac{5}{6}}\int_0^t\Big(\frac{1}{|Q_n|^{\frac{q}{3}}}\int_{Q_n}|u|^2~dx\Big)^{\frac{3}{2}}~ds
\endgathered}
Plugging the above estimate together with pressure estimate in the equation \eqref{5LEI} we get
\EQ{
\gathered
\Big(\frac{1}{|Q_n|^{\frac{q}{3}}}\int_{Q_n}|u|^2(x,t)~dx\Big)^{\frac{r}2}+\Big(\frac{1}{|Q_n|^{\frac{q}{3}}}\int_0^t\int_{Q_n}|\nabla u|^2~dx~dt\Big)^{\frac{r}2}\le
\\
\le \Big(\frac{1}{|Q_n|^{\frac{q}{3}}}\int_{Q_n^*}|u(x,0)|^2~dx\Big)^{\frac{r}2}+\Big(\frac{1}{|Q_n|^{\frac{q+2}{3}}}\int_0^t\int_{Q^*_n}|u|^2~dx~ds\Big)^{\frac{r}2}+
\\
+\Big(C(\varepsilon)|Q_n|^{\frac{2q-4}{3}}\int_0^t\Big(\frac{1}{|Q_n|^{\frac{q}{3}}}\int_{Q_n^{**}} |u|^2~dx\Big)^3~ds\Big)^{\frac{r}2}+
\\
\varepsilon^{\frac{r}2}\Big(\frac{1}{|Q_n|^{\frac{q}{3}}}\int_0^t\int_{Q_n^{**}}|\nabla u|^2~dx~ds\Big)^{\frac{r}2}+
\\
+\Big(|Q_n|^{\frac{q-5}{6}}\int_0^t\Big(\frac{1}{|Q_n|^{\frac{q}{3}}}\int_{Q_n^{**}} |u|^2~dx\Big)^{\frac32}~ds\Big)^{\frac{r}2}+
\\
\Big(\frac{C}{|Q_n|^{\frac{5-q}{6}}}\int_0^t\Big(\sum_{Q'\subset (Q^{**})^{c}}\Big(\frac{1}{|Q'|^{\frac{q}{3}}}\int_{Q'}|u|^2~dy\Big)^{\frac{r}2}\Big)^{\frac{3}{r}}~ds\Big)^{\frac{r}2}
\endgathered }
Next we use Holder inequality in time:
\EQ{
\gathered
\Big(\frac{1}{|Q_n|^{\frac{q}{3}}}\int_{Q_n}|u|^2(x,t)~dx\Big)^{\frac{r}2}+\Big(\frac{1}{|Q_n|^{\frac{q}{3}}}\int_0^t\int_{Q_n}|\nabla u|^2~dx~dt\Big)^{\frac{r}2}\le
\\
\le \Big(\frac{1}{|Q_n|^{\frac{q}{3}}}\int_{Q_n^*}|u(x,0)|^2~dx\Big)^{\frac{r}2}+
t^{\frac{r-2}2}\int_0^t\Big(\frac{1}{|Q_n|^{\frac{q+2}{3}}}\int_{Q^*_n}|u|^2~dx\Big)^{\frac{r}2}~ds+
\\
+C(\varepsilon)|Q_n|^{\frac{r(q-2)}{3}}t^{\frac{r-2}2}\int_0^t\Big(\frac{1}{|Q_n|^{\frac{q}{3}}}\int_{Q_n^{**}} |u|^2~dx\Big)^{\frac{3r}2}~ds+
\\
\varepsilon^{\frac{r}2}\Big(\frac{1}{|Q_n|^{\frac{q}{3}}}\int_0^t\int_{Q_n^{**}}|\nabla u|^2~dx~ds\Big)^{\frac{r}2}+
\\
+|Q_n|^{\frac{r(q-5)}{12}}t^{\frac{r-2}{2}}\int_0^t\Big(\frac{1}{|Q_n|^{\frac{q}{3}}}\int_{Q_n^{**}} |u|^2~dx\Big)^{\frac{3r}{4}}~ds+
\\
\frac{C}{|Q_n|^{\frac{r(5-q)}{12}}}t^{\frac{r-2}{2}}\int_0^t\Big(\sum_{Q'\in \mathcal{C}}\Big(\frac{1}{|Q'|^{\frac{q}{3}}}\int_{Q'}|u|^2~dy\Big)^{\frac{r}2}\Big)^{\frac{3}{2}}~ds
\endgathered }
All of the cubes $Q_n$ have radius at least $2^N$ and we substitute $|Q_n|$ by $2^{3N}$ in the following terms:
\EQ{\gathered
\Big(\frac{1}{|Q_n|^{\frac{q+2}{3}}}\int_{Q^*_n}|u|^2~dx\Big)^{\frac{r}2}\le \frac{C}{2^{rN}}\Big(\frac{1}{|Q_n|^{\frac{q}{3}}}\int_{Q^*_n}|u|^2~dx\Big)^{\frac{r}2},
\\
|Q_n|^{\frac{r(q-2)}{3}}t^{\frac{r-2}2}\int_0^t\Big(\frac{1}{|Q_n|^{\frac{q}{3}}}\int_{Q_n^{**}} |u|^2~dx\Big)^{\frac{3r}2}~ds\le
\\
\frac{C}{2^{rN(2-q)}}t^{\frac{r-2}2}\int_0^t\Big(\frac{1}{|Q_n|^{\frac{q}{3}}}\int_{Q_n^{**}} |u|^2~dx\Big)^{\frac{3r}2}~ds,
\\
|Q_n|^{\frac{r(q-5)}{12}}t^{\frac{r-2}{2}}\int_0^t\Big(\frac{1}{|Q_n|^{\frac{q}{3}}}\int_{Q_n^{**}} |u|^2~dx\Big)^{\frac{3r}{4}}~ds\le
\\
\frac{C}{2^{\frac{rN(5-q)}{4}}}t^{\frac{r-2}{2}}\int_0^t\Big(\frac{1}{|Q_n|^{\frac{q}{3}}}\int_{Q_n^{**}} |u|^2~dx\Big)^{\frac{3r}{4}}~ds.
\endgathered
}
Every $Q_n^{**}$ intersects with a fixed amount of cubes from the family $\mathcal{C}_N$, which is independent of $n$, therefore we can substitute integrals over $Q_n^{**}$ with sum of integrals over $Q_n'$ such that $Q_n'\cap Q_n^{**}\neq\emptyset$. Therefore after taking sum over all $Q_n$ we get the following estimate:
\EQ{\label{5LEI_1}
\gathered
\|u(\cdot,t)\|_{F_r^q}^r+\sum_{n=1}^{\infty}\Big(\frac{1}{|Q_n|^{\frac{q}{3}}}\int_0^t\int_{Q_n}|\nabla u|^2~dx~dt\Big)^{\frac{r}2}\le \|u(\cdot,0)\|_{F_r^q}^r+
\\
\frac{C}{2^{rN}}t^{\frac{r-2}{2}}\int_0^t\|u(\cdot,s)\|^r_{F_r^q}~ds+C(\varepsilon)\frac{C}{2^{rN(2-q)}}t^{\frac{r-2}{2}}\int_0^t\sum_{n=1}^{\infty}
\Big(\frac{1}{|Q_n|^{\frac{q}{3}}}\int_{Q_n} |u|^2~dx\Big)^{\frac{3r}2}~ds+
\\
+C\varepsilon^{\frac{r}2}\sum_{n=1}^{\infty}\Big(\frac{1}{|Q_n|^{\frac{q}{3}}}\int_0^t\int_{Q_n}|\nabla u|^2~dx~ds\Big)^{\frac{r}2}+
\\
\frac{C}{2^{\frac{rN(5-q)}{4}}}\sum_{n=1}^{\infty}t^{\frac{r-2}{2}}\int_0^t\Big(\frac{1}{|Q_n|^{\frac{q}{3}}}\int_{Q_n^{**}} |u|^2~dx\Big)^{\frac{3r}{4}}~ds+
\\
\sum_{n=1}^{\infty}\frac{C}{|Q_n|^{\frac{r(5-q)}{12}}}t^{\frac{r-2}{2}}\int_0^t\Big(\sum_{Q'\in \mathcal{C}}\Big(\frac{1}{|Q'|^{\frac{q}{3}}}\int_{Q'}|u|^2~dy\Big)^{\frac{r}2}\Big)^{\frac{3}{2}}~ds.
\endgathered }
Next we choose $\varepsilon$ sufficiently small to cancel the gradient term and apply the following algebraic inequalities:
\EQ{
\sum_{n=1}^{\infty}\Big(\frac{1}{|Q_n|^{\frac{q}{3}}}\int_{Q_n} |u|^2~dx\Big)^{\frac{\alpha r}2}\le\Big(\sum_{n=1}^{\infty}\Big(\frac{1}{|Q_n|^{\frac{q}{3}}}\int_{Q_n} |u|^2~dx\Big)^{\frac{r}2}\Big)^{\alpha}
}
Where $\alpha$ equals to $3$ or $\frac32$, we can also calculate the multiplier $\sum_{n=1}^{\infty}\frac{C}{|Q_n|^{\frac{r(5-q)}{12}}}\approx \frac{C}{2^{\frac{rN(5-q)}{4}}}$ and plugging it into \eqref{5LEI_1} we get:
\EQ{
\gathered
\|u(\cdot,t)\|_{F_{r}^{q,N}}^r+\sum_{n=1}^{\infty}\Big(\frac{1}{|Q_n|^{\frac{q}{3}}}\int_0^t\int_{Q_n}|\nabla u|^2~dx~dt\Big)^{\frac{r}2}\le
\\
\|u(\cdot,0)\|_{F_r^{q,N}}^r+\frac{C}{2^{rN}}t^{\frac{r-2}{2}}\int_0^t\|u(\cdot,s)\|^r_{F_r^{q,N}}~ds+
\\
Ct^{\frac{r-2}{2}}\int_0^t\frac{1}{2^{rN(2-q)}}\|u(\cdot,s)\|^{3r}_{F_r^{q,N}}+\frac{1}{2^{\frac{rN(5-q)}{4}}}\|u(\cdot,s)\|^{\frac{3r}2}_{F_r^q}~ds\le  \|u(\cdot,0)\|_{F_r^{q,N}}^r+
\\
\frac{C_0t^{\frac{r-2}{2}}}{2^{rN}}\int_0^t\|u(\cdot,s)\|^r_{F_r^{q,N}}~ds+\frac{C_1t^{\frac{r-2}{2}}}{2^{rN(2-q)}}\int_0^t\|u(\cdot,s)\|^{3r}_{F_r^{q,N}}~ds.
\endgathered }
In the last line we also used Young inequality and $C_0,C_1$ are global constants.
\end{proof}

\section{Global Existence}\label{5GEsection}

The idea of the proof relies on getting a-priori estimates with help of a variation of Gronwall which we will apply to the norm of $u$ in spaces $F^{q,m}_r$. We use the same lemma as in \cite{BKT}:
\begin{lemma}\label{5Gronwall}
  Suppose $f(t)\in L^{\infty}([0,T];[0,\infty))$ satisfies, for some $m\ge 1$,
  \EQ{
  f(t)\le a+\int_{0}^{t}(b_1f(s)+b_2f(s)^m)~ds~~\text{for}~ 0<t<T,
  }
  where $a,b_1,b_2\ge 0$ then for $T_0=\min(T,T_1)$, with
  $$
  T_1=\frac{a}{2ab_1+(2a)^mb_2}
  $$
  we have $f(t)\le 2a$ for $t\in(0,T_0).$
\end{lemma}
\begin{remark}\label{5timeT}
  In our case we will apply this lemma for arbitrary large $T$,$m=3$ and $f(t)=\|u(\cdot,t)\|^r_{F^{q,N}_r},~a=f(0),~b_1=\frac{C_0T^{\frac{r-2}{2}}}{2^{rN}}, b_2=\frac{C_1T^{\frac{r-2}{2}}}{2^{rN(2-q)}}$, which will guarantee the uniform apriori bound up to the time $\min\{T_1,T\}$, where
$$
T_1=\frac{1}{2\frac{C_0T^{\frac{r-2}{2}}}{2^{rN}}+(2\|u_0\|^r_{F^{q,N}_r})^2\frac{C_1T^{\frac{r-2}{2}}}{2^{rN(2-q)}}}
$$
We can take $T=n$ for some large integer $n$ and find $N$ such that $T_1>n$, as will be shown in the next lemma. For eventual regularity Theorem \ref{5Regularity} we will take $T=2^{2N}$ and therefore
 $$
T_1=\frac{1}{C_02^{-2N}+C_12^{rN(q-1)-2N}(2\|u_0\|^r_{F^{q,N}_r})^2}\approx 2^{2N},~\text{if}~q \le 1
$$
\end{remark}

\begin{lemma}\label{5uniform_bound}
  Assume $u_0\in F^{q}_r$ is divergence-free, and let $(u,p)$ be a local energy solution with initial data $u_0$ on $\Bbb R^3\times (0,\infty)$, we also assume that $(u,p)$ satisfies the following for all $N$:
  \EQ{
  \|u(\cdot,t)\|_{F_{r}^{q,N}}^r+\sum_{m=1}^{\infty}\Big(\frac{1}{|Q_m|^{\frac{q}{3}}}\int_0^t\int_{Q_m}|\nabla u|^2~dx~dt\Big)^{\frac{r}2}\le\infty \quad \forall t<\infty
  }
  Where $Q_m$ are cubes in $\mathcal{C}_N$. Then for any $n\in\Bbb N$ there exists $N$ and such that
  \EQ{\label{5UB}
  \|u(\cdot,t)\|_{F_{r}^{q,N}}^r+\sum_{m=1}^{\infty}\Big(\frac{1}{|Q_m|^{\frac{q}{3}}}\int_0^t\int_{Q_m}|\nabla u|^2~dx~dt\Big)^{\frac{r}2}\le c_0\|u_0\|_{F_{r}^{q,N}}^r \quad \forall t<n
  }
  Here $c_0$ is a global constant that does not depend on $n$.
\end{lemma}
\begin{proof}
  We will combine Lemma \ref{5aprrox} and Lemma \ref{5Gronwall}, take some integer $n$ and $T=n$ in Lemma \ref{5aprrox}. From local energy estimate for all $N$ we have the following
    \EQ{
\gathered
\|u(\cdot,t)\|_{F_{r}^{q,N}}^r+\sum_{n=1}^{\infty}\Big(\frac{1}{|Q_n|^{\frac{q}{3}}}\int_0^t\int_{Q_n}|\nabla u|^2~dx~dt\Big)^{\frac{r}2}
\le  \|u(\cdot,0)\|_{F_r^{q,N}}^r +
\\
\frac{C_0n^{\frac{r-2}{2}}}{2^{rN}}\int_0^t\|u(\cdot,s)\|^r_{F_r^{q,N}}~ds+\frac{C_1n^{\frac{r-2}{2}}}{2^{rN(2-q)}}\int_0^t\|u(\cdot,s)\|^{3r}_{F_r^{q,N}}~ds
\endgathered }
  For all $t\le n$, therefore we can apply Gronwall lemma with $f=\|u(\cdot,t)\|_{F_r^{q,N}}$.  Since we know that $u_0\in F^q_r$ and $\underset{N\rightarrow\infty}\lim \|u_0\|_{F^{q,N}_r}=0$ we can find $N$ such that
  $$
T_1=\frac{1}{2\frac{C_0n^{\frac{r-2}{2}}}{2^{rN}}+(2\|u_0\|^r_{F^{q,N}_r})^2\frac{C_1n^{\frac{r-2}{2}}}{2^{rN(2-q)}}}>n
$$
Here $T_1$ is given to us by Lemma \ref{5Gronwall} and therefore we have a uniform bound \eqref{5UB}
\end{proof}

\begin{lemma}\label{5approx}
  Assume $f\in F^q_r$ is divergence free then for any $\varepsilon>0$ there exists a divergence free $g\in L^2$ such that $\|f-g\|_{F^{q}_r}\le\varepsilon$ and $\|f-g\|_{F^{q,m}_r}\le\varepsilon$ for any $m\in\Bbb N$
\end{lemma}
\begin{proof}
  Let $\varepsilon>0$ and $f\in F^q_r$ is divergence free, we choose $R>0$ such that $\|f\chi_{\Bbb R^3\setminus B_{R}}\|_{F^q_r}\le\varepsilon$. Then we can find smooth cut-off radial function $\phi\le 1$ such that $\supp \phi\subset B_{2R}$ and $\phi=1$ on $B_R$. Then we can use Bogovskii map \cite{BogovskiiBibl} to construct function $h$ such that
  \EQ{\label{5bogovski}
  \div h=-f\nabla\phi,~\supp h\subset B_{2R}\setminus B_R,~\|h\|_{W^{1,2}}\le \|f\nabla\phi\|_{2, B_{2R}\setminus B_R}.
  }
  Moreover, we can estimate the following norm:
  \EQ{\label{5b1}
  \int_{B_{2R}\setminus B_R}|h|^2~dx\le C_0R^2 \int_{B_{2R}\setminus B_R}|f|^2|\nabla\phi|^2~dx\le C_0 \int_{B_{2R}\setminus B_R}|f|^2~dx
  }
  For some universe constant $C_0$. Next we choose $g=f\phi+h$, from \eqref{5bogovski}we get that $\div g=0$. Now we only need to check that $\|f-g\|_{F^{q}_r}$ is small, indeed
   \EQ{\gathered
  \|f-g\|_{F^{q}_r}\le \|f\chi_{\Bbb R^3\setminus B_R}\|_{F^q_r}+\|h\|_{F^q_r}\le
    \\
  \|f\chi_{\Bbb R^3\setminus B_R}\|_{F^q_r}+\big(\sum_{Q\in\mathcal{C}|~Q\cap B_{2R}\setminus B_R\neq\emptyset}(\frac{1}{|Q|^{q/3}}\int_Q|h|^2~dx)^{r/2}\big)^{\frac{1}{r}}.
    \endgathered}
  Notice that there are only finite amount of cubes in $\mathcal{C}$ intersecting $B_{2R}\setminus B_R$ and the amount does not depend on $R$ due to properties of $C$ moreover all such cubes have comparable radius to $R$, therefore for some universe constant $C$ we have the following
  \EQ{\gathered
  \sum_{Q\in\mathcal{C}|~Q\cap B_{2R}\setminus B_R\neq\emptyset}(\frac{1}{|Q|^{q/3}}\int_Q|h|^2~dx)^{r/2}\le (\frac{C}{|B_R|^{q/3}}\int_{B_{2R}\setminus B_R} |h|^2~dx)^{r/2}\le
   \\
   (\frac{C}{|B_R|^{q/3}}\int_{B_{2R}\setminus B_R} |f|^2~dx)^{r/2}\le C\|f\chi_{\Bbb R^3\setminus B_R}\|_{F^q_r}\le C\varepsilon.
  \endgathered}
   Here we used \eqref{5b1}. Therefore $g$ satisfies all the conditions of the lemma.
\end{proof}
Next we will prove the main theorem:
\begin{proof}
Take $n\in \Bbb N$, from Lemma \ref{5approx} there exists $u_0^n\in L^2(\Bbb R^3)$ such that $\|u_0-u_0^n\|_{F^{q}_r}\le\frac{1}{n}$ and $\|u_0-u_0^n\|_{F^{q,m}_r}\le\frac{1}{n}$ for any $m\in\Bbb N$. Let $(u^n,\bar{p}^n)$ be a global Leray solution for initial data $u_0^n$. By Lemma \ref{5uniform_bound} there exists a sequence $N_n$ such that $N_n>n$ and $u^n$ is bounded uniformly on $B_{N_n}\times[0,n]$. Hence, there exists a sub-sequence $u^{1,k}$ that converges on $B_{N_1}\times (0,1)$ in the following sense:
\EQ{\gathered
u^{1,k}\overset{*}{\rightharpoonup} u_1~\textrm{in} ~L^{\infty}(0,1;L^2(B_{N_1})),
\\
u^{1,k}\rightharpoonup u_1~\textrm{in} ~L^{\infty}(0,1;H^1(B_{N_1})),
\\
u^{1,k}\rightarrow u_1~\textrm{in} ~L^{3}(0,1;L^3(B_{N_1})).
\endgathered}
By Lemma \ref{5uniform_bound} the sequence $u^{1,k}$ is also uniformly bounded on $B_{N_n}\times[0,n]$ for any $n\in\Bbb N$. Therefore by induction we construct a subsequence $\{u^{n,k}\}_{k\in\Bbb N}$ from $\{u^{n-1,k}\}_{k\in\Bbb N}$ which converges to a vector field $u_n$ on $B_{N_n}\times(0,n)$ as $k\rightarrow\infty$ in the following sense:
\EQ{\gathered
u^{n,k}\overset{*}{\rightharpoonup} u_n~\textrm{in} ~L^{\infty}(0,n;L^2(B_{N_n})),
\\
u^{n,k}\rightharpoonup u_n~\textrm{in} ~L^{\infty}(0,n;H^1(B_{N_n})),
\\
u^{n,k}\rightarrow u_n~\textrm{in} ~L^{3}(0,n;L^3(B_{N_n})).
\endgathered}
Denote $\tilde{u_n}$ as a $0$ extension of $u_n$  to $\Bbb R^3\times (0,\infty)$. From our construction $\tilde{u_n}=\tilde{u_{n-1}}$ on $B_{N_{n-1}}\times(0,n-1)$. Let $u=\lim_{n\rightarrow\infty} \tilde{u_n}$. Then, $u=u_n$ on $B_{N_n}\times(0,n)$ for every $n\in\Bbb N$.

Let $u^{(k)}=u^{k,k}$ on $B_{N_k}\times(0,k)$ and $0$ elsewhere. Then for any $n\in\Bbb N$ we have the following as $k\rightarrow\infty$:
\EQ{\gathered\label{5ukconv}
u^{(k)}\overset{*}{\rightharpoonup} u~\textrm{in} ~L^{\infty}(0,n;L^2(B_{N_n})),
\\
u^{(k)}\rightharpoonup u~\textrm{in} ~L^{\infty}(0,n;H^1(B_{N_n})),
\\
u^{(k)}\rightarrow u~\textrm{in} ~L^{3}(0,n;L^3(B_{N_n})).
\endgathered}
Using Lemma \ref{5uniform_bound} we have the following bound for $u$:
 \EQ{
  \sup_{0<t<n}\|u(\cdot,t)\|_{F_{r}^{q,N}}^r+\sum_{m=1}^{\infty}\Big(\frac{1}{|Q_m|^{\frac{q}{3}}}\int_0^n\int_{Q_m}|\nabla u|^2~dx~dt\Big)^{\frac{r}2}\le c_0\|u_0\|_{F_{r}^{q,N}}^r
  }
The pressure is dealt similarly to \cite{BK}, indeed we know that there is weak limit $p^{(k)}\rightharpoonup p$ in $L^{3/2}((0,n)\times Q_{N_n})$, we will check the pressure decomposition formula and strong convergence. Take $Q\in \mathcal{C},~T>0$, then for any $k$ $(u^{(k)},p^{(k)})$ are global Leray energy solutions and there exists $p_Q^{(k)}$ such that for all $(x,t)\in Q\times[0,T]$:
\EQ{\gathered
p^{(k)}(x,t)-p_Q^{(k)}=G_{ij}^Q(u^{(k)}_iu^{(k)}_j)=-\frac13\delta_{ij}u^{(k)}_iu^{(k)}_j+
\\
\textrm{p.v.}\int_{y\in Q^{**}}K_{ij}(x-y)(u_i^{(k)}u_j^{(k)})(y)~dy+
\\
+\textrm{p.v.}\int_{y\notin Q^{**}}(K_{ij}(x-y)-K_{ij}(x_Q-y))(u_i^{(k)}u_j^{(k)})(y)~dy
\endgathered}
Choose $m>T$ so that $Q^{**}\subset Q_m=Q(0,2^m)$ and separate $G_{ij}^Q(u_iu_j)$ into two terms
\EQ{\gathered
p_1^{(k)}(x,t)=-\frac13\delta_{ij}u^{(k)}_iu^{(k)}_j+\textrm{p.v.}\int_{y\in Q^{**}}K_{ij}(x-y)(u_i^{(k)}u_j^{(k)})(y)~dy+\\
+\textrm{p.v.}\int_{Q_m\setminus Q^{**}}(K_{ij}(x-y)-K_{ij}(x_Q-y))(u_i^{(k)}u_j^{(k)})(y)~dy
\\
p_2^{(k)}(x,t)=\textrm{p.v.}\int_{y\notin Q_m}(K_{ij}(x-y)-K_{ij}(x_Q-y))(u_i^{(k)}u_j^{(k)})(y)~dy
\endgathered}
We denote $p_1,p_2$ as similar terms but where $u^{(k)}$ is substituted by $u$. To prove convergence we choose $\varepsilon>0$ arbitrary small and estimate the difference
\EQ{\gathered
|p_2-p_2^{(k)}|\le \int_{y\notin Q_m}(K_{ij}(x-y)-K_{ij}(x_Q-y))|u_i^{(k)}u_j^{(k)}-u_iu_j|(y)~dy\le
\\
\sum_{\tilde{Q}\in\mathcal{C}~|\tilde{Q}|^{1/3}\ge 2^{m-1}}\int_{\tilde{Q}}(K_{ij}(x-y)-K_{ij}(x_Q-y))|u_i^{(k)}u_j^{(k)}-u_iu_j|(y)~dy
\le
\\
\sum_{\tilde{Q}\in\mathcal{C}~|\tilde{Q}|^{1/3}\ge 2^{m-1}}\int_{\tilde{Q}}\frac{|Q|^{1/3}}{|\tilde{Q}|^{4/3}}|u_i^{(k)}u_j^{(k)}-u_iu_j|(y)~dy.
\endgathered}
Here we used that $Q^{**}\subset Q_m$ and therefore $|K_{ij}(x-y)-K_{ij}(x_Q-y)|\le C\frac{|Q|^{1/3}}{|\tilde{Q}|^{4/3}}$. Next we will apply Holder inequality
\EQ{\gathered
\sum_{\tilde{Q}\in\mathcal{C}~|\tilde{Q}|^{1/3}\ge 2^{m-1}}\int_{\tilde{Q}}\frac{|Q|^{1/3}}{|\tilde{Q}|^{4/3}}|u_i^{(k)}u_j^{(k)}-u_iu_j|(y)~dy\le
\\
|Q|^{1/3}\Big(\sum_{\tilde{Q}\in\mathcal{C}~|\tilde{Q}|^{1/3}\ge 2^{m-1}}|\tilde{Q}|^{\frac{(q-4)r}{3(r-2)}}\Big)^{\frac{r-2}{r}}
\\
\Big( \sum_{\tilde{Q}\in\mathcal{C}~|\tilde{Q}|^{1/3}\ge 2^{m-1}}\Big(\frac{1}{|\tilde{Q}|^{q/3}}\int_{\tilde{Q}}|u_i^{(k)}u_j^{(k)}-u_iu_j|(y)~dy\Big)^\frac{r}{2}\Big)^{2/r}\le
\\
|Q|^{1/3}\Big(\sum_{\tilde{Q}\in\mathcal{C}~|\tilde{Q}|^{1/3}\ge 2^{m-1}}|\tilde{Q}|^{\frac{(q-4)r}{3(r-2)}}\Big)^{\frac{r-2}{r}}(\|u\|_{F_r^{q,m}}^2+\|u^{(k)}\|_{F_r^{q,m}}^2)\le
\\
C|Q|^{1/3}2^{m(q-4)}(\|u_0\|_{F_r^{q,m}}^2+\|u^{(k)}_0\|_{F_r^{q,m}}^2).
\endgathered}
Here we used \eqref{5UB}, since we have uniform bound for $u,u^{(k)}$ up to time $T$ with our choice of $m$. Notice that $q<4$ and $u^{(k)}_0$ converges to $u_0$, therefore we choose $m,k$ large enough so that $\|u^{(k)}_0\|_{F_r^{q,m}}^2\le \varepsilon+\|u_0\|_{F_r^{q,m}}^2$ and
\EQ{
C|Q|^{1/3}2^{m(q-4)}(\|u_0\|_{F_r^{q,m}}^2+\|u^{(k)}_0\|_{F_r^{q,m}}^2)\le 2\varepsilon
}
Next we apply \eqref{5ukconv} to get large enough $k$ so that
\EQ{
\|p_1^{(k)}-p^1\|_{L^{3/2}(Q\times[0,T])}\le \varepsilon
}
This proves that $G_{ij}^Q(u^{(k)}_iu^{(k)}_j)\rightarrow G_{ij}^Q(u_iu_j)$ in $L^{3/2}(Q\times[0,T])$. Lastly, in any fixed domain $Q\times[0,T]$ pair $(u,G^Q_{ij}(u_iu_j)$ solves \eqref{5NS}. Therefore, $\nabla p=\nabla G_{ij}^Q(u_iu_j)$ in $\mathcal{D}'(\Bbb R^3)$ at every time $t$ and so there exists a function of time $p_Q(t)$ such that
$$
p(x,t)= G_{ij}^Q(u_iu_j)(x)+p_Q(t), ~\text{for} ~(x,t)\in Q\times [0,T]
$$
Also from this identity we get $p_Q(t)\in L^{\frac{3}{2}}(0,T)$. The rest of the proof follows from the similar argument to \cite{BK}, because we have the same convergence of $u^k,p^k$ on any $Q\times T_0$ for any fixed $T_0>0$.

 \end{proof}

\section{Eventual regularity}\label{5regularsection}
In this Section we will prove Theorem \ref{5Regularity} , we will use one of Caffarelli-Kohn-Nirenberg epsilon-regularity criteria variations, see \cite{CKN}
\begin{lemma}\label{5epsilon}
  For any $\sigma\in(0,1)$, there exists a universal constant $\varepsilon_*(\sigma)>0$ such that, if a pair $(u,p)$ is a suitable weak solution of \eqref{5NS} in $Q_r=B_r(x_0)\times(t_0-r^2,t_0),$ and satisfies the bound
\EQ{
\varepsilon^3=\frac{1}{r^2}\int_{Q_r}(|u|^3+|p|^{\frac32})~dx~dt<\varepsilon_*,
}
Then $u\in L^{\infty}(Q_{\sigma r}).$ Moreover, there is $L^{\infty}$ estimate
\EQ{
\|\nabla^k u\|_{L^{\infty}(Q_{\sigma r})}\le C_k\varepsilon r^{-k-1},~\forall k\in\Bbb Z_+
}
for universal constants $C_k=C_k(\sigma)$.
\end{lemma}
Next we will prove Theorem \eqref{5Regularity}
\begin{proof}
  Choose sufficiently large $T,N_1$ such that $\|u_0\|_{F^{q,m}_r}<1$ for all $m> N_1$ and Lemma \ref{5uniform_bound} holds for $N>N_1$ and $t<T$. We take some $N>N_1$ and denote the following:
  \EQ{
  Q=(-2^N,2^N)^3.
  }
  Note that $Q\in\mathcal{C}_N$. From Lemmas \ref{5pressure},\ref{5uniform_bound} we have that
  \EQ{\gathered
  J=\frac{1}{|Q|^{\frac23}}\int_0^t\int_Q(|u|^3+|p-(p)_Q(s)|^{\frac{3}{2}}~dx~dx\le
  \\
  \le \frac{C}{|Q|^{\frac{1-q}2}}\int_0^t\Big(\sum_{Q'\cap (Q^{**})^{c}\neq \emptyset}\Big(\frac{1}{|Q'|^{\frac{q}{3}}}\int_{Q'}|u|^2~dy\Big)^{\frac{r}2}\Big)^{\frac{3}{r}}~ds
+\frac{C}{|Q|^{\frac{2}3}}\int_0^t\int_{Q^{**}}|u|^3~dx~ds,
  \endgathered}
 here cubes $Q'$ are part of family $\mathcal{C}_N$. Next we apply \eqref{5convectiveterm}:
 \EQ{\gathered
\frac{1}{|Q|^{\frac23}}\int_0^t\int_{Q}|u|^3~dx~dt\le C(\varepsilon)|Q|^{q-\frac53}\int_0^t\Big(\frac{1}{|Q|^{\frac{q}{3}}}\int_{Q}|u|^2~dx\Big)^3~dt+
\\
\frac{\varepsilon}{|Q|^{\frac{1}{3}}}\int_0^t\int_{Q}|\nabla u|^2~dx~ds+C|Q|^{\frac{q}{2}-\frac{7}{6}}\int_0^t\Big(\frac{1}{|Q|^{\frac{q}{3}}}\int_{Q}|u|^2~dx\Big)^{\frac{3}{2}}~ds
\endgathered}
Since $q\le1$ we can apply Lemma \ref{5uniform_bound} to get the following:
\EQ{
J=\frac{1}{|Q|^{\frac23}}\int_0^t\int_Q(|u|^3+|p-(p)_Q(s)|^{\frac{3}{2}}~dx~dx\le \|u_0\|_{F^{1,r}_N},~\forall t<T.
}
Take $\sigma<1$ from Lemma \ref{5epsilon} arbitrary small. From Lemma \ref{5intialdata} we can choose $m$ large enough so that $\|u_0\|_{F^{1,r}_m}\le\varepsilon(\sigma)$. Then by Remark \ref{5timeT} we can take $T=c_*2^m$ with $c_*<1$ and get that
\EQ{
\frac{1}{c_*2^{2m}}\int_{0}^{c_*2^{2m}}\int_{B_{2^m}(0)}(|u|^3+|p-(p)_Q|^{\frac{3}{2}})~dx~dt\le \varepsilon(\sigma)
}
By epsilon-regularity Lemma \ref{5epsilon} we get that $u$ is regular in
\EQ{
Z_m=B_{\sigma \sqrt{c_*}2^{m}}\times [(1-\sigma^2)c_*2^{2m},c_*2^{2m}],
}
and $\|u\|_{L^{\infty}(Z_m)}\le C2^{-m}$. Take $\delta<1$ and choose $c_*,\sigma$ so that $Z_m$ contains
\EQ{
P_m=\{(x,t)\in \Bbb R^4_+:~ \delta |x|^2\le t,~~(1-\sigma^2)c_* 2^{2m}\le t\le 4(1-\sigma^2) c_*2^{2m}\}.
}
After taking the union by $m\ge N_1$ we prove the statement of \eqref{5Regularity}.

\end{proof}


\begin{thebibliography}{99}

\bibitem{Leray}
{\sc J.~Leray}, {\it Sur le mouvement d’un liquide visqueux emplissant l’espace}, Acta Math. 63 (1934),
no. 1, 193–248.

\bibitem{Basson}
{\sc A.~Basson}, {\it Solutions spatialement homogenes adaptees au sens de Caffarelli, Kohn et Nirenberg des equations de Navier-Stokes}, Universite d\'Evry, 2006.

\bibitem{BogovskiiBibl}
{\sc Bogovski˘ı, M.E.} , {\it Solutions of some problems of vector analysis, associated with the operators div and grad, Theory of
cubature formulas and the application of functional analysis to problems of mathematical physics}, Trudy Sem. S. L.
Soboleva, no. 1, vol. 1980, Akad. Nauk SSSR Sibirsk. Otdel., Inst. Mat., Novosibirsk, 1980, pp. 5–40, 149.
\bibitem{CKN}
{\sc L. Caffarelli, R. Kohn, and L. Nirenberg}, {\it Partial regularity of suitable weak solutions of the
Navier-Stokes equations}, Comm. Pure Appl. Math. 35 (1982), no. 6, 771–831.

\bibitem{LR}
{\sc P.~G.~Lemari\'{e}-Rieusset},{\it Recent developments in the Navier-Stokes problem,}  Chapman \&
Hall/CRC Research Notes in Mathematics, vol. 431, Chapman \& Hall/CRC, Boca Raton,
FL, 2002.
\bibitem{KiSe}
{\sc N.~ Kikuchi,    G.~ Seregin }, {\it Weak solutions to the Cauchy problem for the Navier-Stokes
equations satisfying the local energy inequality }, Nonlinear equations and spectral theory, Amer.
Math. Soc. Transl. Ser. 2, vol. 220, Amer. Math. Soc., Providence, RI, 2007, pp. 141–164.
\bibitem{RS}
{\sc W.~ Rusin, V.~Sverak }, {\it Minimal initial data for potential Navier-Stokes singularities},J. Funct. Anal. 260 (2011), no. 3, 879–891.

\bibitem{JiaSverak-minimal}
{\sc H.~ Jia, V. ~Sverak}, {\it Minimal $L^3$-initial data for potential Navier-Stokes singularities },SIAM
J. Math. Anal. 45 (2013), no. 3, 1448–1459.

\bibitem{BT8}
{\sc Z.~Bradshaw and T.-P.~Tsai}, {\it Global existence, regularity, and uniqueness of infinite energy solutions to the Navier-Stokes equations},
Commun. in Partial Differential Equations 45 (2020), no. 9, 1168-1201.

\bibitem{BKT}

{\sc T.-P.~ Tsai, Z. ~Bradshaw, I. ~Kukavica }, {\it Existence of global weak solutions to the Navier-Stokes equations in weighted spaces},Indiana Univ. Math. J. 71 (2022), no. 1, 191-212

\bibitem{FDLR}

{\sc P.~G.~ Fernandez-Dalgo, P.~G.~ Lemarie-Rieusset}, {\it Weak solutions for Navier-Stokes equations
with initial data in weighted $L^2$ spaces}, arxiv:1906.11038.

\bibitem{BK}
{\sc Z. ~Bradshaw, I. ~Kukavica}, {\it Existence of suitable weak solutions to the Navier-Stokes equa-
tions for intermittent data}, J. Math. Fluid Mech.

\bibitem{Berg_Lefstrem}

 {\sc J.~Bergh, J.~L\" ofstr\" om}, Interpolation spaces. An
introduction. Springer, 1976.

\bibitem{Tsutsui}

{\sc Y.~ Tsutsui},{\it The Navier-Stokes equations and weak Herz spaces} Adv. Differential Equations 16
(2011), no. 11-12, 1049–1085.

\bibitem{CSTY_1}
{\sc Ch.-Ch.~Chen, R.M.~Strain, T.-P.~ Tsai, H.-T.~ Yau}, {\it Lower
bounds on the blow-up rate of the axisymmetric Navier-Stokes
equations}, International Mathematics Research Notices  (2008), no.
9.

\bibitem{CSTY_2}
{\sc Ch.-Ch.~Chen, R.M.~Strain, T.-P.~ Tsai, H.-T.~ Yau}, {\it Lower
bounds on the blow-up rate of the axisymmetric Navier-Stokes
equations II}, Communications in Partial Differential Equations, 34
(2009), no. 3, 203--232.


\bibitem{Evans}
{\sc L.C.~Evans}, {\it Partial Differential Equations}, Second
Edition, AMS, 2010.


\bibitem{Fri_Vicol}
{\sc S.~Friedlander, V.~Vicol}, {\it Global well-posedness for an
advection-diffusion equation arising in magneto-geostrophic
dynamics}, Ann. Inst. H. Poincare Anal. Non Lineaire, 28 (2011), no.
2, 283--301.

\bibitem{Gi_Ra}
{\sc V.~Girault, P.-A.~Raviart}, Finite element approximation of the
Navier-Stokes equations, Springer, 1979.

\end{thebibliography}
\end{document}